\documentclass[sn-mathphys]{sn-jnl}

\jyear{2021}%

\theoremstyle{thmstyleone}%
\newtheorem{theorem}{Theorem}
\newtheorem{lemma}[theorem]{Lemma}

\theoremstyle{thmstyletwo}%

\theoremstyle{thmstylethree}%

\begin{document}

\title[Normalized ground state solutions of nonlinear Schr\"odinger equations]{Normalized ground state solutions of nonlinear Schr\"odinger equations involving exponential critical growth}


\author*[1]{\fnm{Xiaojun} \sur{Chang}}\email{changxj100@nenu.edu.cn}

\author[1]{\fnm{Manting} \sur{Liu}}\email{liumt679@nenu.edu.cn}
\equalcont{These authors contributed equally to this work.}

\author[2]{\fnm{Duokui} \sur{Yan}}\email{duokuiyan@buaa.edu.cn}
\equalcont{These authors contributed equally to this work.}

\affil*[1]{\orgdiv{School of Mathematics and Statistics \& Center for Mathematics and Interdisciplinary Sciences}, \orgname{Northeast Normal University},  \city{Changchun}, \postcode{130024},  \country{China}}

\affil[2]{\orgdiv{School of Mathematical Sciences}, \orgname{Beihang University},  \city{Beijing}, \postcode{100191}, \country{China}}


\abstract{We are concerned with the following nonlinear Schr\"odinger equation
 \begin{eqnarray*}
\begin{aligned}
\begin{cases}
-\Delta u+\lambda u=f(u) \ \ {\rm in}\ \mathbb{R}^{2},\\
u\in H^{1}(\mathbb{R}^{2}),~~~ \int_{\mathbb{R}^2}u^2dx=\rho,
\end{cases}
\end{aligned}
\end{eqnarray*}
where $\rho>0$ is given, $\lambda\in\mathbb{R}$ arises as a Lagrange multiplier and $f$ satisfies an exponential critical growth. Without assuming the Ambrosetti-Rabinowitz condition, we show the existence of normalized ground state solutions for any $\rho>0$. The proof is based on a constrained minimization method and the Trudinger-Moser inequality in $\mathbb{R}^2$.}




\maketitle

\section{Introduction and main results}\label{sec1}

In this paper, we investigate the following stationary nonlinear Schr\"odinger equation
\begin{gather}\label{1.1}
\begin{aligned}
\begin{cases}
-\Delta u+\lambda u=f(u) \ \ {\rm in}\ \mathbb{R}^{N},\\
\int_{\mathbb{R}^N}u^2dx=\rho,
\end{cases}
\end{aligned}
\end{gather}
where $\rho>0$ and $\lambda\in\mathbb{R}$ is a Lagrange multiplier. If $(u,\lambda)\in H^{1}(\mathbb{R}^{N})\times \mathbb{R}$ solves (\ref{1.1}), then $u$ is often called a normalized solution of (\ref{1.1}).

The study of (\ref{1.1}) is related to the search of standing wave solutions with form $\psi(t,x)=e^{-i\lambda t}u(x)$ for the following time-dependent nonlinear Schr\"odinger equation
\begin{equation}\label{Schrodinger}
i\frac{\partial \psi}{\partial t}-\Delta \psi=g(\vert\psi\vert^2)\psi\ \ \mathrm{in}\ \mathbb{R}^N,
\end{equation}
which appears in many models such as nonlinear optics and Bose-Einstein condensations (see \cite{AA1999, BC2013, EGBB1997, F2010}).
Based on the law of conservation of mass, the integral $\int_{\mathbb{R}^N}\vert\psi(x,t)\vert^2dx$ is prescribed along the time evolution of (\ref{Schrodinger}). Accordingly, particular attention is drawn to standing wave solutions with a prescribed $L^2$ norm.
From a physical point of view, the prescribed mass can be interpreted as power supply in nonlinear optics, while it can also represents the total number of the particles in the Bose-Einstein condensates.

 Define the energy functional $J: H^{1}(\mathbb{R}^{N})\rightarrow\mathbb{R}$ corresponding to (\ref{1.1}) by
\begin{eqnarray}\label{FunctionalJ}
J(u):= \frac{1}{2}\int_{\mathbb{R}^{N}}\vert\nabla u\vert^{2}dx-\int_{\mathbb{R}^{N}}F(u)dx,
\end{eqnarray}
where $F(u):=\int_{0}^{u}f(s)ds$. Set
\begin{eqnarray*}
\mathcal{S}_{\rho}:= \left\{u\in H^{1}(\mathbb{R}^{N}):\int_{\mathbb{R}^{N}}u^2dx=\rho\right\}.
\end{eqnarray*}
Clearly, any critical point of $J$ on $\mathcal{S}_{\rho}$ is a normalized solution of (\ref{1.1}).
If $f$ admits a $L^2$ subcritical growth, i.e., $f$ has a growth $|u|^{p-1}$ with $p<2_*:= 2+\frac{4}{N}$, then $J\vert_{\mathcal{S}_{\rho}}$ is bounded below and one can use the minimization method to get a global minimizer \cite{BL1983-1,CL1982,HS2005,S1988}. If $f$ admits a $L^2$ supercritical growth, i.e., $p>2_*$, then $J\vert_{\mathcal{S}_{\rho}}$ is unbounded below. Consequently, the direct minimization on $\mathcal{S}_{\rho}$ does not work and different approaches are required. Indeed, in 1997, Jeanjean \cite{Jeanjean1997} developed a mountain-pass argument for the scaled functional $\tilde{J}\left(u,s\right):=J\left(s\star u\right)$ with $s\star u(\cdot):= e^{\frac{Ns}{2}}u(e^s\cdot)$ to study the $L^2$ supercritical problems. Subsequently, Bartsch and de Valerioda \cite{BDeV2013} applied the fountain theorem to $\tilde{J}$ and obtained infinitely many normalized solutions of (\ref{1.1}). In \cite{IT2019}, Ikoma and Tanaka established a deformation result for the $L^2$ normalized solutions and gave an alternative proof of the results in \cite{BDeV2013,Jeanjean1997}. We note that the following conditions are important in these studies:
 \begin{description}
\item[($h_1$)]~$f: \mathbb{R}\to \mathbb{R}$ is continuous and odd;
\item[($h_2$)]~$N\ge2$, and there exist constants $\eta_1, \eta_2\in \mathbb{R}$ with $2_*<\eta_1<\eta_2<2^*(2^*:= \frac{2N}{N-2}$ for $N\ge3$ and $2^*:=+\infty$ for $N=2$) such that
\begin{eqnarray}\label{AR-1}
0<\eta_1 F(t)\le f(t)t\le \eta_2 F(t),  \qquad \forall t\in \mathbb{R}\setminus\{0\}.
\end{eqnarray}
\end{description}
Here the first part of (\ref{AR-1}) is the usual Ambrosetti-Rabinowitz condition, which is an essential assumption in above papers to obtain bounded Palais-Smale sequences for $J$ constrained on $\mathcal{S}_{\rho}$.

In \cite{BS2017-1,BS2017-2}, Bartsch and Soave presented a minimax method based on the Ghoussoub minimax
principle in \cite{G-1993} to obtain the existence and multiplicity results of (\ref{1.1}). Besides $(h_1), (h_2)$, they also need the following assumption:
 \begin{description}
\item[($h_3$)]~$H$ is of class $C^1$ and $h(t)t>2_*H(t)$,
\end{description}
 where $H(t):=f(t)t-2F(t)$ and $h(t)=H'(t)$.

In \cite{JL2020}, Jeanjean and Lu studied the normalized solutions of (\ref{1.1}) by replacing the Ambosetti-Rabinowitz condition with the following conditions:
 \begin{description}
 \item[($h_4$)]~$\lim\limits_{\vert t\vert\to+\infty}\frac{F(t)}{\vert t\vert^{2_*}}=+\infty$;
\item[($h_5$)]~$\frac{H(t)}{\vert t\vert^{2_*}}$ is strictly decreasing on $(-\infty,0)$ and strictly increasing on $(0,+\infty)$.
\end{description}
By introducing a minimax argument for the case when $f$ is only continuous and analyzing the behavior of ground state energy, they proved the existence of normalized ground state solutions of (\ref{1.1}) for all $N\ge1$. Specially, when $N=2$, they considered the case that $f$ has an exponential subcritical growth, that is, for all $\alpha>0$,
\begin{eqnarray}
\lim\limits_{\vert t\vert \to+\infty}\frac{\vert f(t)\vert}{e^{\alpha t^2}}=0.
\end{eqnarray}

In \cite{BM2020}, Bieganowski and Mederski developed a minimization approach to study the existence of normalized ground state solutions of (\ref{1.1}) when $N\ge3$ and $f$ is $L^2$ supercritical but Sobolev subcritical.
The Ambrosetti-Rabinowitz condition was not assumed in \cite{BM2020}, and the results were established under the following condition:
 \begin{description}
 \item[($h_6$)]~$h(t)t\succ2_{*}H(t)$,
\end{description}
which means that $h(t)t\ge 2_{*}H(t)$ for all $t\in \mathbb{R}$, and for any $\gamma>0$ there is $|t|<\gamma$ such that $h(t)t>2_{*}H(t)$.
 Later, in \cite{MS-2022}, this approach is also effective in the case that $f$ satisfies the Sobolev critical growth.

 Recently, Alves, Ji and Miyagaki \cite{AJM2021,AJM2021-2}
  used some minimax arguments as in \cite{Jeanjean1997,JL2020} combined with the Trudinger-Moser inequality in $\mathbb{R}^2$ (see \cite{C1992,do1997}) to study the existence and multiplicity of normalized solutions to
   (\ref{1.1}) in the case when $N=2$ and $f$ admits an exponential critical growth, i.e.,
there exists $\alpha_0>0$ such that
\begin{eqnarray}\label{critical}
\lim\limits_{\vert t\vert\to+\infty}\frac{\vert f(t)\vert}{e^{\alpha t^2}}=0, \, \, \forall \alpha>\alpha_0,  \qquad \lim\limits_{\vert t\vert\to+\infty}\frac{\vert f(t)\vert}{e^{\alpha t^2}}=+\infty,  \, \, \forall \alpha<\alpha_0.
\end{eqnarray}
In particular, in \cite{AJM2021}, the authors proved the existence of normalized ground state solutions to (\ref{1.1}) under the assumption (\ref{critical}) with $\alpha_0=4\pi$ for $\rho\in(0,1)$ and some suitable conditions on $f$, where the Ambrosetti-Rabinowitz condition plays a crucial role.

It is natural to ask if one can show the existence of normalized ground state solutions to (\ref{1.1}) for the exponential critical nonlinearity without the Ambrosetti-Rabinowitz condition in the case of $N=2$. The goal of this paper is to give an affirmative answer to this question. Inspired by \cite{BM2020,JL2020,MS-2022}, for any $\rho>0$, we will apply a minimization argument to prove the existence of normalized ground state solutions to (\ref{1.1}) under a weak monotonicity condition.

For more studies on normalized solutions of nonlinear Schr\"odinger equations, we refer to \cite{JJLV-2022,JL-MA-2021,L-2021,S2020-1,S2020-2,WW-2022} and the references therein.

In the sequel, we assume $N=2$ in (\ref{1.1}). We call $u_{\rho}\in H^1(\mathbb{R}^2)$ a normalized ground state solution of (\ref{1.1}), if there is $\lambda\in \mathbb{R}$ such that $(u_{\rho},\lambda)$ satisfies (\ref{1.1}) and
\begin{equation*}
J(u_{\rho})=\inf\{J(u): u\in \mathcal{S}_{\rho}, \left(J\vert_{\mathcal{S}_{\rho}}\right)'(u)=0\}.
\end{equation*}

Throughout the paper, the following assumptions are formulated:
\begin{description}
\item[($f_1$)]~$f\in C^1(\mathbb{R}, \mathbb{R})$ and (\ref{critical}) holds;
\item[($f_2$)]~$\lim\limits_{\vert t\vert\to 0}\frac{f(t)}{\vert t\vert^{3}}=0$;
\item[($f_3$)]~$h(t)t\ge 4H(t), \quad \forall t\in \mathbb{R}$;
\item[($f_4$)]~there exist $p>4$ and $\beta>0$ such that
\begin{eqnarray*}
sgn(t)f(t)\ge \beta \vert t\vert^{p-1}, \quad \forall t\in \mathbb{R},
\end{eqnarray*}
where $sgn: \mathbb{R}\setminus\{0\}\to \mathbb{R}$ is defined by
\begin{eqnarray*} sgn(t)=\left\{
\begin{array}{ll}
1 ~~&\mbox{if}~~t>0,\\
-1 ~~&\mbox{if}~~t<0.
\end{array}
\right.
\end{eqnarray*}
\end{description}

Let $H_r^1(\mathbb{R}^2)$ denotes the subspace of radial symmetric functions in $H^1(\mathbb{R}^2)$. Our main result is stated as follows.
\begin{theorem}\label{Thm1.1}
For any $\rho>0,$ there exists $\beta_{1}>0$ such that the following holds: if $f$ satisfies $(f_{1})- (f_{4})$ with $\beta\ge \beta_{1},$ then (\ref{1.1}) admits a normalized ground state solution $u^*\in H^1_r(\mathbb{R}^2)$ for some $\lambda\in \mathbb{R}^+$.
\end{theorem}

Set
\begin{eqnarray}\label{Mdef}
\mathcal{M}:=\{u\in H^{1}(\mathbb{R}^{2})\backslash \{0\}: \int_{\mathbb{R}^{2}}\vert\nabla u\vert^{2}dx=\int_{\mathbb{R}^{2}}H(u)dx\}.
\end{eqnarray}
Clearly, $\mathcal{M}$ contains all the nontrivial solutions of (\ref{1.1}). To prove Theorem \ref{Thm1.1}, \textcolor{blue}{as in} \cite{BM2020,MS-2022}, we will transform the search of minimizers for $J$ on $\mathcal{S}_{\rho} \cap \mathcal{M}$ into looking for minimizers for $J$ on $\mathcal{D}_{\rho} \cap \mathcal{M}$, where
\begin{eqnarray}\label{Drho}
\mathcal{D}_{\rho}:=\left\{u\in H^{1}(\mathbb{R}^{2})\setminus\{0\}: \int_{\mathbb{R}^{2}}u^{2}dx\leq \rho \right\}.
\end{eqnarray}
Since the nonlinearity is of exponential critical growth, it is challenging to recover the compactness of the minimizing sequences $\{u_n\}$ of $J$ on $\mathcal{D}_{\rho}\cap \mathcal{M}$. Due to the absence of the Ambrosetti-Rabinowitz condition, the main obstacle is to control the bound of $\{\|\nabla u_n\|_2\}$ under $(f_3)$. We will overcome this difficulty by introducing new estimates on the ground state energy. Using the new estimates, we can combine with the Trudinger-Moser inequality and the Strauss compactness lemma (see \cite{BL1983-1}) to obtain the important convergences of $F(u_n)\to F(u_0)$ and $f(u_n)u_n\to f(u_0)u_0$ in $L^1(\mathbb{R}^2)$. We note that in \cite{AJM2021}, the authors used the Trudinger-Moser inequality and a variant of the Lebesgue dominated convergence Theorem to show these convergences, where the assumptions $\alpha_0=4\pi$ in \eqref{critical} and $\rho\in(0,1)$ act as a vital part. However, our argument works for all $\alpha_0>0$ and $\rho>0$.

The paper is organized as follows. In Section 2, we introduce some preliminary
results. In Section 3, we give the proof of Theorem \ref{Thm1.1}. For simplicity, in the sequel, we denote by $C_1,  C_2, \cdots$ the positive constants that may be different in different places. Let $\|\cdot\|_{r}$ denote the usual norm of the Lebesgue space $L^r(\mathbb{R}^2)$ for $r\in[1,+\infty)$.

\section{Preliminaries}\label{sec2}
Firstly, we recall the following Trudinger-Moser inequality in $\mathbb{R}^2$ (see \cite{C1992,do1997}).
\begin{lemma}\label{TM}
If $\alpha>0$ and $u\in H^{1}(\mathbb{R}^{2})$, then
\begin{eqnarray*}
\int_{\mathbb{R}^{2}}\left(e^{\alpha u^2}-1\right)dx<\infty.
\end{eqnarray*}
Furthermore, if $\|\nabla u\|_2\le1$, $\|u\|_2\le M_1<+\infty$ with $M_1>0$ and $\alpha <4\pi$, then there exists a constant $C>0$, which depends only on $M_1$ and
$\alpha$, such that
 \begin{eqnarray*}
\int_{\mathbb{R}^{2}}\left(e^{\alpha u^2}-1\right)dx\le C\left(M_1,\alpha\right).
\end{eqnarray*}
\end{lemma}
By \cite{W1983}, we have the following Gagliardo-Nirenberg inequality in $\mathbb{R}^2$.
\begin{lemma}\label{GN}
For any $r>2,$ there exists a best constant $C_{r,2}>0$ such that
\begin{eqnarray*}
\|u\|_{r}^{r}\le C_{r,2}\|u\|_2^2\|\nabla u\|_2^{r-2}, \forall u\in H^{1}(\mathbb{R}^{2}).
\end{eqnarray*}
\end{lemma}

The following observation will be useful.
\begin{lemma}\label{mono} Assume that $(f_1)-(f_3)$ hold. Then
\begin{description}
\item[(i)] $\frac{H(t)}{t^{4}}$ is nonincreasing on $(-\infty,0)$ and nondecreasing on $(0,+\infty)$;
\item[(ii)]$g(t):=\frac{f(t)t-4F(t)}{t^2}$ is nonincreasing on $(-\infty,0)$ and nondecreasing on $(0,+\infty)$;
\item[(iii)] $f(t)t\ge4F(t), \forall t\in \mathbb{R}$.
\end{description}
\end{lemma}
\begin{proof}
 Define $g_1(t):= \frac{H(t)}{t^4}$. Recall that $H(t)=f(t)t-2F(t)$ is $C^1$ and $h(t)=H'(t)$, it follows that $g'_1(t)=\frac{h(t)t-4H(t)}{t^5}$. Since $f\in C^1$, by $(f_3)$ we get $(i)$.
Using $g'(t)=\frac{h(t)t-4H(t)}{t^3}$, we obtain $g'(t)\le 0$ for $t<0$, $g'(t)\ge0$ for $t>0$. Hence (ii) holds.
For (iii), we define
 \begin{eqnarray*}g_2(t)= \left\{
\begin{array}{ll}
\frac{f(t)t-4F(t)}{t^2}~~&\mbox{for} ~~ t\neq 0,\\
0~~&\mbox{for}~~~t=0.
\end{array}
\right.
\end{eqnarray*}
 By $(f_1)-(f_3)$, we can see that $g_2$ is a continuous function on $\mathbb{R}$. In view of (ii), it follows that $g_2(t)\ge 0, \forall t\in \mathbb{R}$, which implies (iii).
\end{proof}

\begin{lemma}\label{asymp}
Assume $(f_1)-(f_4)$ hold. Set $s\star u(\cdot) = e^{s}u(e^s \cdot)$. The functional $J$ is defined by \eqref{FunctionalJ}. Then, for any $u\in H^1(\mathbb{R}^2)\setminus \{0\}$, we have
 \begin{description}
\item[(i)]$J\left(s\star u\right)\to 0^+$ as $s\to-\infty$;
\item[(ii)]$J\left(s\star u\right)\to-\infty$ as $s\to+\infty$.
\end{description}
\end{lemma}
\begin{proof}~
For (i), by $(f_1)-(f_2)$, there exist $q>p$ and $\alpha>\alpha_0$ such that, for any $\epsilon>0$, there exists $C_\epsilon>0$ such that
    \begin{eqnarray*}\label{5-6-1}
\vert F(t)\vert\le \epsilon \vert t\vert^{4}+C_{\epsilon}\vert t\vert^{q+1} \left(e^{\alpha \vert t\vert^{2}}-1\right),   \  \forall t\in \mathbb{R}.
    \end{eqnarray*}
Then, by $(e^s-1)^r\le e^{sr}-1$ for $r>1$ and $s\ge0$, and the H\"older inequality, we get
\begin{eqnarray*}
\int_{\mathbb{R}^{2}}\vert F(u)\vert dx&\le& \epsilon\int_{\mathbb{R}^{2}}\vert u\vert^4dx+C_{\epsilon}\int_{\mathbb{R}^{2}}\vert u\vert^{q+1} \left(e^{\alpha \vert u\vert^{2}}-1\right)dx\\
&\le&\epsilon\int_{\mathbb{R}^{2}}\vert u\vert^4dx+C_{\epsilon}\left(\int_{\mathbb{R}^2}\vert u\vert^{2q+2}\right)^{\frac{1}{2}}\left(\int_{\mathbb{R}^{2}}\left(e^{2\alpha \vert u\vert^{2}}-1\right)dx\right)^{\frac{1}{2}}.
    \end{eqnarray*}
Define $v(x):= \sqrt{\frac{\alpha}{\pi}}e^{s}u(e^{s}x)$. Then
\begin{eqnarray*}
\int_{\mathbb{R}^{2}}\vert\nabla v\vert^{2}dx=e^{2s}\frac{\alpha}{\pi}\int_{\mathbb{R}^{2}}\vert\nabla u\vert^{2}dx\to 0~~ \ \mathrm{as}\ s\to-\infty
 \end{eqnarray*}
 and
 $$
 \int_{\mathbb{R}^{2}}\vert v\vert^{2}dx=\frac{\alpha}{\pi}\int_{\mathbb{R}^{2}}\vert u\vert^{2}dx.
 $$
 By Lemma \ref{TM}, there exists $s_0<0$ with $\vert s_0\vert$ large enough such that, for all $s\le s_0$,
\begin{eqnarray*}
\int_{\mathbb{R}^2}\left(e^{2\alpha \vert e^{s}u(e^{s}x)\vert^2}-1\right)dx=\int_{\mathbb{R}^2}\left(e^{2\pi \vert v\vert^2}-1\right)dx\le C_1
\end{eqnarray*}
for some $C_1>0$. Hence, for $s\le s_0$,
 \begin{eqnarray}
    		\int_{\mathbb{R}^{2}}\vert F(e^{s}u(e^{s}x))\vert dx \le \epsilon e^{2s}\int_{\mathbb{R}^{2}}\vert u(x)\vert^{4}dx+C_{\epsilon}C_1^{\frac{1}{2}} e^{qs}\left(\int_{\mathbb{R}^{2}}\vert u(x)\vert^{2q+2}dx\right)^{\frac{1}{2}}.\label{5-3-1}
    \end{eqnarray}	
Then
 $$
    \begin{aligned}
    	J\left(s\star u\right)
    	&=\frac{1}{2}\int_{\mathbb{R}^{2}}\vert\nabla(e^{s}u(e^{s}x))\vert^{2}dx-\int_{\mathbb{R}^{2}}F(e^{s}u(e^{s}x))dx\\
    	&\ge \frac{1}{2}e^{2s}\int_{\mathbb{R}^{2}}\vert \nabla u\vert^{2}dx-\epsilon e^{2s}\int_{\mathbb{R}^{2}}\vert u(x)\vert^{4}dx-C_{\epsilon}C_1^{\frac{1}{2}} e^{qs}\left(\int_{\mathbb{R}^{2}}\vert u(x)\vert^{2q+2}dx\right)^{\frac{1}{2}}.
    \end{aligned}
    $$
By $q>p>4$, taking $\epsilon>0$ small if necessary, we can obtain that there exists $s_1\le s_0$ such that $J\left(s\star u\right)\ge0$ for $s\le s_1$. Furthermore, by (\ref{5-3-1}) it follows that
    \begin{eqnarray*}
    		\int_{\mathbb{R}^{2}}F\left(e^{s}u(e^{s}x)\right)dx\to 0~~\mbox{as}~s\to -\infty,
    \end{eqnarray*}	
which implies that $J\left(s\star u\right)\to0^+$ as $s\to-\infty$.

 For (ii), by $(f_1)-(f_2)$ and $(f_4)$, there exist $r>4, t_0\in (0, 1], C_2, C_3>0$ such that
 \begin{eqnarray*}
&&F(t)\ge C_2\vert t\vert^{r}, \forall \vert t\vert\ge t_0,\\
&&F(t)\le C_3\vert t\vert^2, \forall \vert t\vert\le t_0.
\end{eqnarray*}
Then, for some $C_4>0$,
\begin{eqnarray*}
\int_{\mathbb{R}^2}F(u)dx&=&\int_{\left\{\vert u(x)\vert\ge t_0\right\}}F(u)dx+\int_{\{\vert u(x)\vert<t_0\}}F(u)dx\\
&\ge&C_2\int_{\{\vert u(x)\vert\ge t_0\}}\vert u\vert^{r}dx-C_3\int_{\{\vert u(x)\vert<t_0\}}\vert u\vert^2dx\\
&=&C_2\int_{\mathbb{R}^2}\vert u\vert^{r}dx-\int_{\{\vert u(x)\vert<t_0\}}\left[C_2\vert u\vert^{r}+C_3\vert u\vert^2\right]dx\\
&\ge&C_2\int_{\mathbb{R}^2}\vert u\vert^{r}dx-C_4\int_{\mathbb{R}^2}\vert u\vert^2dx,
\end{eqnarray*}
which implies that
 \begin{eqnarray*}
 J\left(s\star u\right)
\le \frac{1}{2}e^{2s}\int_{\mathbb{R}^{2}}\vert\nabla u\vert^{2}dx+C_4\int_{\mathbb{R}^{2}}\vert u\vert^{2}dx-C_2e^{(r-2)s}\int_{\mathbb{R}^{2}}\vert u\vert^{r}dx.
    \end{eqnarray*}
By $r>4,$ it follows that $J\left(s\star u\right)\to-\infty$ as $s\to +\infty$.
  \end{proof}

Define the Nehari-Pohozaev functional
\begin{eqnarray*}
\mathcal{P}(u) :=\int_{\mathbb{R}^{2}}\vert \nabla u\vert^{2}dx-\int_{\mathbb{R}^{2}}H(u)dx, \qquad \forall u\in H^1(\mathbb{R}^2).
\end{eqnarray*}
We have the following results.
\begin{lemma}\label{pro}Assume $(f_1)-(f_3)$ hold. Then
 \begin{description}
\item[(i)]for any $u\in H^1(\mathbb{R}^2)\setminus \{0\}$, there exists a number $s_u\in \mathbb{R}$ such that $\mathcal{P}\left(s_u\star u\right)=0$ and
\begin{eqnarray}
J\left(s_u\star u\right)\ge J\left(s\star u\right), \qquad \forall s\neq s_u.
\end{eqnarray}
Specially, if $u\in \mathcal{M}$, then $J(u)=\max\limits_{s\in \mathbb{R}}J\left(s\star u\right)$;
\item[(ii)]$\mathcal{D}_{\rho}\cap\mathcal{M}\neq\emptyset$;
\item[(iii)]there exists $\delta_0>0$ such that $\inf\limits_{u\in \mathcal{D}_{\rho}\cap\mathcal{M}}\|\nabla u\|_2\ge\delta_0$.
\end{description}
\end{lemma}
\begin{proof}
By Lemma \ref{asymp}, we can see that $\phi_u(s):=J\left(s\star u\right)$ admits a global maximum at some $s_u\in \mathbb{R}$. In view of $\frac{d}{ds}J\left(s\star u\right)=\mathcal{P}\left(s\star u\right)$, it follows that (i) holds. Take $u\in \mathcal{D}_{\rho}$, where $\mathcal{D}_{\rho}$ is given by \eqref{Drho}. By (i), there is $s_u\in \mathbb{R}$ such that $s_u\star u \in \mathcal{M}$, where is $\mathcal{M}$ given by \eqref{Mdef}. Note that $s_u\star u \in \mathcal{D}_{\rho}$, we get (ii).

For (iii), by contradiction, we may assume that there exists a sequence $\{u_n\}\subset \mathcal{D}_{\rho}\cap\mathcal{M}$ such that $\|\nabla u_n\|_2\to 0$ as $n\to+\infty$. Then, by Lemma \ref{TM}, there exists $C_1>0$ such that
$$
\int_{\mathbb{R}^{2}}\left(e^{2\alpha \vert u_n\vert^{2}}-1\right)dx\le C_1.
$$
Using similar arguments as in Lemma \ref{asymp}, by Lemma \ref{GN}, there exist $q>p, C_2>0$, and for any $\epsilon>0$ a $C_{\epsilon}>0$ such that
    \begin{eqnarray*}
\int_{\mathbb{R}^{2}}\vert f(u_n)u_n\vert dx&\le&\epsilon \int_{\mathbb{R}^{2}}\vert u_n\vert^{4}dx +C_{\epsilon}\left(\int_{\mathbb{R}^2}\vert u_n\vert^{2q+2}\right)^{\frac{1}{2}}\left(\int_{\mathbb{R}^{2}}\left(e^{2\alpha \vert u_n\vert^{2}}-1\right)dx\right)^{\frac{1}{2}}\\
&\le&\epsilon \, C_2\rho \int_{\mathbb{R}^{2}}\vert \nabla u_n\vert^{2}dx+C_{\epsilon}C_1^{\frac{1}{2}}C_2 \rho^{\frac{1}{2}} \left(\int_{\mathbb{R}^{2}}\vert \nabla u_n\vert^{2}dx\right)^{\frac{q}{2}}.
    \end{eqnarray*}
    Similarly,
        \begin{eqnarray*}
\int_{\mathbb{R}^{2}}\vert F(u_n)\vert dx\le\epsilon \, C_2\rho \int_{\mathbb{R}^{2}}\vert \nabla u_n\vert^{2}dx+C_{\epsilon}C_1^{\frac{1}{2}}C_2 \rho^{\frac{1}{2}} \left(\int_{\mathbb{R}^{2}}\vert \nabla u_n\vert^{2}dx\right)^{\frac{q}{2}}.
    \end{eqnarray*}
    Then, taking $\epsilon=\frac{1}{4C_2\rho}$, we get
$$
\begin{aligned}
    	\int_{\mathbb{R}^{2}}\vert \nabla u_n\vert^{2}dx
    	&=\int_{\mathbb{R}^{2}}\left[f(u_{n})u_n-2F(u_{n})\right]dx\\
    	   &\le \frac{1}{2}\int_{\mathbb{R}^{2}}\vert \nabla u_n\vert^{2}dx+2C_{\epsilon}C_1^{\frac{1}{2}}C_2 \rho^{\frac{1}{2}} \left(\int_{\mathbb{R}^{2}}\vert \nabla u_n\vert^{2}dx\right)^{\frac{q}{2}},
    	\end{aligned}
    	$$
which implies that $\frac{1}{2}\leq C_{\epsilon}C_1^{\frac{1}{2}}C_2 \rho^{\frac{1}{2}} \|\nabla u_{n}\|_{2}^{q-2}$.
By $q>4$, we obtain a contradiction. Hence (iii) holds.
\end{proof}

Denote
\begin{eqnarray*}
&&c_{\rho}:= \min\limits_{u\in\mathcal{D}_{\rho} \cap \mathcal{M}}J(u),\\
&&c_{\rho,r}:= \min\limits_{u\in\mathcal{D}_{\rho} \cap \mathcal{M} \cap H_r^1(\mathbb{R}^2)}J(u).
\end{eqnarray*}
\begin{lemma}\label{radial}
Assume that $(f_1)-(f_3)$ hold. Then $c_{\rho}=c_{\rho,r}$.
\end{lemma}
\begin{proof}
For any $u\in \mathcal{D}_{\rho}$ and $s\in \mathbb{R}$, by Lemma \ref{mono} (ii), we have
\begin{eqnarray*}
J\left(s\star u\right)-\frac{1}{2}\mathcal{P}\left(s\star u\right)&=&\frac{1}{2}\int_{\mathbb{R}^2}\left[f\left(e^su(e^sx)\right)e^su(e^sx)-4F\left(e^su(e^sx)\right)\right]dx\\
&=&\frac{1}{2}\int_{\mathbb{R}^2}\frac{f(e^su)e^su-4F(e^su)}{e^{2s}u^2}u^2dx.
\end{eqnarray*}
Clearly, $J\left(s\star u\right)-\frac{1}{2}\mathcal{P}\left(s\star u\right)$ is nondecreasing with respect to $s\in \mathbb{R}$.

Take a minimizing sequence $\{u_n\}$ of $J$ on $\mathcal{D}_{\rho}\cap\mathcal{M}$ and denote the Schwarz symmetrization of $\{u_n\}$ by $\{u_n^*\}$. Then
\begin{eqnarray*}
&&\int_{\mathbb{R}^2}\vert \nabla  u_n^*\vert^2dx\le \int_{\mathbb{R}^2}\vert \nabla  u_n\vert^2dx, \\
&&\int_{\mathbb{R}^2}\vert u_n^*\vert^2dx=\int_{\mathbb{R}^2}\vert u_n\vert^2dx,\\
&& \int_{\mathbb{R}^2}F(u_n^*)dx=\int_{\mathbb{R}^2}F(u_n)dx, \int_{\mathbb{R}^2}f(u_n^*)u_n^*dx=\int_{\mathbb{R}^2}f(u_n)u_ndx.
\end{eqnarray*}
It follows that $u_n^*\in \mathcal{D}_{\rho}\cap H_r^1(\mathbb{R}^2)$ and $\mathcal{P}(u_n^*)\le \mathcal{P}(u_n)=0$.
By $(f_3)$, there exists $s_n^*:= s_n^*(u_n^*)\le 0$ such that $\mathcal{P}(s_n^*\star u_n^*)=0$. Hence, by Lemma \ref{pro}, we get
\begin{eqnarray*}
c_{\rho}\le c_{\rho,r}\le J(s_n^*\star u_n^*)&=& J(s_n^*\star u_n^*)-\frac{1}{2}\mathcal{P}(s_n\star u_n^*)\\
&\le& J(u_n^*)-\frac{1}{2}\mathcal{P}(u_n^*)\\
&=&\frac{1}{2}\int_{\mathbb{R}^2}\left[f(u_n^*)u_n^*-4F(u_n^*)\right]dx\\
&=&\frac{1}{2}\int_{\mathbb{R}^2}\left[f(u_n)u_n-4F(u_n)\right]dx=J(u_n)=c_{\rho}+o_n(1),
\end{eqnarray*}
which implies that $c_{\rho}=c_{\rho,r}$.
\end{proof}

\begin{lemma}\label{minimizer}
Suppose that $(f_1)-(f_3)$ hold.
Assume that $\{u_{n}\}\subset H^{1}_{r}(\mathbb{R}^{2})$ is a bounded minimizing sequence of $J$ on $\mathcal{D}_{\rho}\cap\mathcal{M}$ and there exists $u_0\in H^{1}_{r}(\mathbb{R}^{2})$ such that $u_n(x)\to u_0(x)$ for a.e. $x\in \mathbb{R}^2$ and
\begin{eqnarray}\label{8-10-1}
&&\int_{\mathbb{R}^{2}}F(u_n)dx\rightarrow\int_{\mathbb{R}^{2}}F(u_0)dx,\\
&&\int_{\mathbb{R}^{2}}f(u_n)u_ndx\rightarrow\int_{\mathbb{R}^{2}}f(u_0)u_0dx.\label{8-10-2}
\end{eqnarray}
Then
$c_\rho$ is achieved by $u_0$ and $c_{\rho}>0$.
\end{lemma}
\begin{proof}
Firstly, we claim that, under $(f_1)-(f_3)$, the following result holds
  \begin{eqnarray}\label{8-11-3}
\Gamma(u):=\int_{\mathbb{R}^2}\left[f(u)u-4F(u)\right]dx>0,  \forall u\in H_r^1(\mathbb{R}^2)\setminus\{0\}.
  \end{eqnarray}
In fact, by the Strauss radial lemma in \cite{BL1983-1}, we may assume that $u$ is continuous. Since $u\in H^1_r(\mathbb{R}^2),$ we get $\vert u(x)\vert\to0$ as $\vert x\vert\to +\infty$. By Lemma \ref{mono} (iii), we have $\Gamma(u)\ge0$. If $\Gamma(u)=0$, then $f(u(x))u(x)-4F(u(x))=0$ for all $x\in \mathbb{R}^2$.
 Then, there exists an open interval $I$ such that $0\in \bar{I}$ and $f(u)u-4F(u)=0$ for $u(x)\in \bar{I}$. By direct computations, we get $F(u)=C\vert u\vert^{4}$ for some $C>0$ and $u\in \bar{I}$. But this is contrary to $(f_1)-(f_2)$. Hence (\ref{8-11-3}) holds.

Since $\{u_n\}\subset \mathcal{D}_{\rho}\cap\mathcal{M}$,
by (\ref{8-10-1})-(\ref{8-10-2}) and the Fatou lemma, we get
   	\begin{eqnarray*}
   		\int_{\mathbb{R}^{2}}u_{0}^{2}dx\leq\liminf_{n\rightarrow +\infty}\int_{\mathbb{R}^{2}}u_{n}^{2}dx\leq \rho
   	\end{eqnarray*}
   	and
  	\begin{eqnarray*}
   	\int_{\mathbb{R}^{2}}\vert \nabla u_{0}\vert^2dx
   	\leq\liminf_{n\rightarrow +\infty}\int_{\mathbb{R}^{2}}\vert \nabla u_{n}\vert^2dx
   	=\liminf_{n\rightarrow +\infty}\int_{\mathbb{R}^{2}}H(u_{n})dx
   	=\int_{\mathbb{R}^{2}}H(u_{0})dx.
   	\end{eqnarray*}
If $\int_{\mathbb{R}^{2}}\vert \nabla u_{0}\vert^2dx<\int_{\mathbb{R}^{2}}H(u_{0})dx$, defining $t_{0}:= \left(\frac{\int_{\mathbb{R}^{2}}\vert \nabla u_{0}\vert^2dx}{\int_{\mathbb{R}^{2}}H(u_{0})dx}\right)^{\frac{1}{2}}$, we have $t_{0}\in(0,1)$.
  Note that
   	\begin{eqnarray*}
   		\mathcal{P}\left(u_{0}(\frac{x}{t_{0}})\right)=\int_{\mathbb{R}^{2}}\vert \nabla u_{0}\vert^2dx-t_{0}^{2}\int_{\mathbb{R}^{2}}H(u_{0})dx=0
   	\end{eqnarray*}
and
   	\begin{eqnarray*}
   		\int_{\mathbb{R}^{2}}\left(u_{0}(\frac{x}{t_{0}})\right)^{2}dx=t_{0}^{2}\int_{\mathbb{R}^{2}}u_{0}^{2}dx<\rho,
   	\end{eqnarray*}
it follows that $u_{0}(\frac{\cdot}{t_{0}})\in\mathcal{D}_{\rho}\cap\mathcal{M}$. By (\ref{8-10-1})-(\ref{8-11-3}) and the Fatou lemma, we obtain
\[
   	\begin{aligned}
c_{\rho}\leq J\left(u_{0}(\frac{x}{t_{0}})\right)
   	&=\frac{1}{2}\int_{\mathbb{R}^{2}}\vert \nabla\left(u_{0}(\frac{x}{t_{0}})\right)\vert^2dx-\int_{\mathbb{R}^{2}}F\left(u_{0}(\frac{x}{t_{0}})\right)dx\\
   	&=\frac{1}{2}\int_{\mathbb{R}^{2}}\vert \nabla u_{0}\vert^2dx-t_{0}^{2}\int_{\mathbb{R}^{2}}F(u_{0})dx\\
   	&=\frac{1}{2}t_{0}^{2}\int_{\mathbb{R}^{2}}H(u_{0})dx-t_{0}^{2}\int_{\mathbb{R}^{2}}F(u_{0})dx\\
   	&=\frac{1}{2}t_{0}^{2}\int_{\mathbb{R}^{2}}\left[H(u_{0})-2F(u_{0})\right]dx\\
   	&<\frac{1}{2}\int_{\mathbb{R}^{2}}\left[H(u_{0})-2F(u_{0})\right]dx\\
   	&=\liminf_{n\rightarrow +\infty}\frac{1}{2}\int_{\mathbb{R}^{2}}\left[H(u_{n})-2F(u_{n})\right]dx\\
   	&=\liminf_{n\rightarrow +\infty}J(u_{n})=c_{\rho},
   	\end{aligned}
\]
which gives a contradiction. Hence $\int_{\mathbb{R}^{2}}\vert \nabla u_{0}\vert^2dx=\int_{\mathbb{R}^{2}}H(u_{0})dx$. Thus $u_0\in\mathcal{D}_{\rho}\cap\mathcal{M}$ and $\int_{\mathbb{R}^{2}}\vert \nabla u_{n}\vert^2dx\to \int_{\mathbb{R}^{2}}\vert \nabla u_{0}\vert^2dx$. Using (\ref{8-10-1}) and Lemma \ref{radial}, we get $J(u_{0})=c_{\rho}$.

Now we prove $c_{\rho}>0$. In fact, by $(f_1)-(f_3)$ and $u_0\in \mathcal{D}_{\rho}\cap \mathcal{M}\cap H_r^1(\mathbb{R}^2)$, using (\ref{8-11-3}) we get
  \begin{eqnarray*}
 c_{\rho}=J(u_0)=\frac{1}{2}\int_{\mathbb{R}^2}\left[f(u_0)u_0-4F(u_0)\right]dx>0.
  \end{eqnarray*}
  This completes the proof.
\end{proof}

\section{Proof of Theorem  \ref{Thm1.1}}\label{sec3}
In this section, we give the proof of Theorem \ref{Thm1.1}. We first introduce some new estimates on the ground state energy and the upper bound of the minimizing sequences of $J$ on $\mathcal{D}_{\rho}\cap \mathcal{M}$.
\begin{lemma}\label{energy-2}
$c_{\rho}\to0$ as $\beta\to+\infty$.
\end{lemma}
\begin{proof}
Define $\psi_0:= \frac{\rho}{\pi}e^{-\vert x\vert^2}, \forall x\in \mathbb{R}^2$. Clearly, $\int_{\mathbb{R}^2}\vert \psi_0\vert^2dx=\rho$, which implies that $\psi_0\in \mathcal{D}_{\rho}$. By $(f_4)$, for any $s\in \mathbb{R}$, we have
\begin{eqnarray*}
J(s\star\psi_0)\le \frac{1}{2}e^{2s}\int_{\mathbb{R}^2}\vert\nabla \psi_0\vert^2dx-\frac{\beta}{p}e^{(p-2)s}\int_{\mathbb{R}^2}\vert \psi_0\vert^pdx.
\end{eqnarray*}
Define $\varphi(t):=\frac{1}{2}A_1t^2-\frac{\beta}{p}A_2t^{p-2}, \forall t\in \mathbb{R}$, where
$$
A_1:= \int_{\mathbb{R}^2}\vert\nabla \psi_0\vert^2dx,~ A_2:= \int_{\mathbb{R}^2}\vert \psi_0\vert^pdx.
$$
Clearly, $A_1, A_2>0$. Then
$$
\varphi'(t)=A_1t-\frac{\beta (p-2)}{p}A_2t^{p-3}$$
 and $$\varphi''(t)=A_1-\frac{\beta (p-2)(p-3)}{p}A_2t^{p-4}.
 $$
 Denote
 $$
 t_1:=\left(\frac{pA_1}{\beta(p-2)A_2}\right)^{\frac{1}{p-4}}.
 $$
 By direct computations, we get $\varphi(0)=\varphi'(0)=\varphi'(t_1)=0$, $\varphi''(t_1)<0$, $\varphi(t)\to -\infty$ as $t\to+\infty$. Hence $\varphi$ admits the unique maximum at $t_1$ and
 $$
 \varphi(t_1)=\beta^{\frac{-2}{p-4}}A_3>0
 $$ with
 $$
 A_3:= A_1^{\frac{p-2}{p-4}}\left(\frac{p}{(p-2)A_2}\right)^{\frac{2}{p-4}}\frac{p-4}{2(p-2)}>0.
 $$
 By Lemma \ref{pro} and $p>4$, we get $c_{\rho}\le \beta^{\frac{-2}{p-4}}A_3 \to 0$ as $\beta\to +\infty$.
\end{proof}

Set $\beta_0:= \left(\frac{8\alpha_0 A_3}{\pi}\right)^{\frac{p-4}{2}}$. In view of the definition of $A_3$, we know $\beta_0$ depends only on $\rho$ and $p$.
Obviously, $\beta^{\frac{-2}{p-4}}A_3\le \frac{\pi}{8\alpha_0}, \forall \beta\ge \beta_0$.
\begin{lemma}\label{upper-bound}
If $\{u_n\}$ is a minimizing sequence of $J$ on $\mathcal{D}_\rho\cap \mathcal{M}$, then
there is $M>0$ independent of $n$ such that $\|\nabla u_n\|_2^2\le M$ for all $\beta\ge \beta_0$.
\end{lemma}
\begin{proof}
We assume by contradiction that $\|\nabla u_n\|_2\to+\infty$ for some $\beta \ge \beta_0$. For any $n\ge1$, define
$$
s_n:= \ln\left(\|\nabla u_n\|_{2}\right),~~v_n:=(-s_n)\star u_n.
$$
Then $s_n\to+\infty$
and
$$
\int_{\mathbb{R}^2}\vert\nabla v_n\vert^2dx=e^{-2s_n}\int_{\mathbb{R}^2}\vert\nabla u_n\vert^2dx=1.
$$
Set $\hat{\delta}:= \limsup\limits_{n\rightarrow +\infty}\big(\sup\limits_{y\in \mathbb{R}^2}\int_{B(y,1)}\vert v_n\vert^2dx\big)$. We will distinguish the two cases:

Case 1. $\hat{\delta}>0$, i.e., non-vanishing occurs. Then, up to a subsequence, there exist $\{y_n\}\subset \mathbb{R}^2$ and $z_0\in H^1(\mathbb{R}^2)$ with $z_0\not\equiv0$ such that
\begin{eqnarray*}
z_n:= v_n\left(\cdot+y_n\right)\rightharpoonup z_0~~\mbox{in}~~H^1(\mathbb{R}^2),~~~z_n\to z_0,~~a.e.~~ x\in H^1(\mathbb{R}^2).
\end{eqnarray*}
By $(f_4)$ and the Fatou lemma, it follows that
\begin{eqnarray*}
c_{\rho}+o_n(1)=J(u_n)&=&J(s_n\star v_n)\\
&=& \frac{1}{2}e^{2s_n}\int_{\mathbb{R}^{2}}\vert\nabla v_n\vert^2dx-e^{-2s_n}\int_{\mathbb{R}^2}F(e^{s_n}v_n)dx\\
&=&e^{2s_n}\left[\frac{1}{2}-e^{-4s_n}\int_{\mathbb{R}^2}F(e^{s_n}z_n)dx\right]\\
&\le&e^{2s_n}\left[\frac{1}{2}-\frac{\beta_0}{p}e^{(p-4)s_n}\int_{\mathbb{R}^2}\vert z_0\vert^pdx\right]\\
&\to&-\infty,
\end{eqnarray*}
which is a contradiction.

Case 2. $\hat{\delta}=0$, i.e., $\{v_n\}$ is vanishing. By the Lions lemma (see \cite{W1996}), $v_n\to 0$ in $L^r(\mathbb{R}^2)$ for all $r>2$.
By $(f_1)-(f_2)$ and similar arguments as in Lemma \ref{asymp}, for any $\epsilon>0$, there exists $C_{\epsilon}>0$ such that
\begin{eqnarray}\label{7-14-1}
\int_{\mathbb{R}^{2}}\vert F(u) \vert dx\le\epsilon\int_{\mathbb{R}^{2}}\vert u\vert^4dx+C_{\epsilon}\left(\int_{\mathbb{R}^2}\vert u\vert^{2q+2}\right)^{\frac{1}{2}}\left(\int_{\mathbb{R}^{2}}\left(e^{4\alpha_0 \vert u\vert^{2}}-1\right)dx\right)^{\frac{1}{2}}.
    \end{eqnarray}
Define $w_n(x):= \sqrt{\frac{2\alpha_0}{\pi}}e^{s}v_n(e^{s}x)$. Then
$$
\int_{\mathbb{R}^{2}}\vert\nabla w_n\vert^{2}dx=e^{2s}\frac{2\alpha_0}{\pi}\int_{\mathbb{R}^{2}}\vert\nabla v_n\vert^{2}dx=e^{2s}\frac{2\alpha_0}{\pi}\le 1, \forall s\le \frac{1}{2}\ln\left(\frac{\pi}{2\alpha_0}\right)
$$
 and
 $$\int_{\mathbb{R}^{2}}\vert w_n\vert^{2}dx=\frac{2\alpha_0}{\pi}\int_{\mathbb{R}^{2}}\vert v_n\vert^{2}dx\le\frac{2\alpha_0 \rho}{\pi}.
 $$
 By Lemma \ref{TM}, for all $s\le \frac{1}{2}\ln\left(\frac{\pi}{2\alpha_0}\right)$,
\begin{eqnarray*}
\int_{\mathbb{R}^2}\left(e^{4\alpha_0 \vert e^{s}v_n(e^{s}x)\vert^2}-1\right)dx=\int_{\mathbb{R}^2}\left(e^{2\pi \vert w_n\vert^2}-1\right)dx\le C_1
\end{eqnarray*}
for some $C_1>0$. Hence, for all $s\le \frac{1}{2}\ln\left(\frac{\pi}{2\alpha_0}\right)$,
 \begin{eqnarray*}
    		\int_{\mathbb{R}^{2}}\vert F(s\star v_n)\vert dx \le \epsilon e^{2s}\int_{\mathbb{R}^{2}}\vert v_n(x)\vert^{4}dx+C_{\epsilon}C_1^{\frac{1}{2}} e^{qs}\left(\int_{\mathbb{R}^{2}}\vert v_n(x)\vert^{2q+2}dx\right)^{\frac{1}{2}}\to0.
    \end{eqnarray*}	
Then, by Lemma \ref{pro} (i), taking $s_0=\frac{1}{2}\ln(\frac{\pi}{2\alpha_0})$, we get
 \begin{eqnarray*}
c_{\rho}+o_n(1)=J(u_n)&\ge&J((s_0-s_n)\star u_n)\\
&=& \frac{1}{2}e^{2s_0}-\int_{\mathbb{R}^2}F(s_0*v_n)dx=\frac{\pi}{4\alpha_0}+o_n(1).
\end{eqnarray*}
Then we get a contradiction to $\beta\ge \beta_0$.
\end{proof}

Furthermore, we have the follow crucial estimate.
\begin{lemma}\label{bound}
If $\{u_n\}$ is a minimizing sequence of $J$ on $\mathcal{D}_\rho\cap \mathcal{M}$, then
there are $n_1\in \mathbb{Z}^+$ and $\beta_1:=\beta(\rho, p)>0$ such that $\|\nabla u_n\|_2\le \frac{\pi}{2\alpha_0}$ for all $n\ge n_1$ and $\beta\ge \beta_1$.
\end{lemma}
\begin{proof}
By Lemma \ref{upper-bound}, $\|\nabla u_n\|_2^2\le M$ for some $M>0$.
Define $s_0$ as in Lemma \ref{upper-bound}, i.e., $s_0:=\frac{1}{2}\ln(\frac{\pi}{2\alpha_0}).$
Define $\tilde{s}_k:= \ln\left(M+1+k\right)$, where $k\ge 1$ is to be determined.
By $(f_1)-(f_2)$, for any $\epsilon>0$, there exist $\delta_1\in(0,1)$ and $\delta_2\in (2,+\infty)$ such that
\begin{eqnarray*}
\vert F(t)\vert\le \epsilon \vert t\vert^4, \, \forall \vert t \vert\le \delta_1,\,\, \vert F(t)\vert \le \epsilon \vert t\vert^{p+1}(e^{2\alpha_0t^2}-1),\, \forall \vert t\vert \ge \delta_2.
\end{eqnarray*}
Then, together with $(f_4)$, we get
\begin{eqnarray}\label{10-20-1}
\vert F(t)\vert\le \epsilon \vert t\vert^4+\epsilon \vert t\vert^{p+1}(e^{2\alpha_0t^2}-1)+C_{\epsilon}\vert t\vert^p, \, \forall t\in \mathbb{R},
\end{eqnarray}
where $C_{\epsilon}:=4\beta+M_{\epsilon}$ with $M_{\epsilon}:=\sup\limits_{\vert t\vert \le \delta_2}\vert F(t)\vert +\sup\limits_{\delta_1\le \vert t\vert \le 1}\frac{\vert F(t)\vert }{\vert t\vert^p}$.

  Set
  $$
  \tilde{w}_{n}(x):= \sqrt{\frac{2\alpha_0}{\pi}}e^{s_0-\tilde{s}_k}u_n\left(e^{s_0-\tilde{s}_k}x\right).
  $$
   Then
\begin{eqnarray*}
&&\int_{\mathbb{R}^{2}}\vert\nabla \tilde{w}_{n}\vert^{2}dx=e^{2(s_0-\tilde{s}_k)}\frac{2\alpha_0}{\pi}\int_{\mathbb{R}^{2}}\vert\nabla u_n\vert^{2}dx\le\frac{\int_{\mathbb{R}^{2}}\vert\nabla u_n\vert^{2}dx}{(M+1)^2}\le 1, \\ &&\int_{\mathbb{R}^{2}}\vert \tilde{w}_{n}\vert^{2}dx\le \frac{2\alpha_0 \rho}{\pi}.
\end{eqnarray*}
By Lemma \ref{TM}, there exists $C_1(\alpha_0, \rho)>0$ such that
\begin{eqnarray*}
\int_{\mathbb{R}^2}\left(e^{4\alpha_0 \vert(s_0-\tilde{s}_k)\star u_n\vert^2}-1\right)dx=\int_{\mathbb{R}^2}\left(e^{2\pi \vert \tilde{w}_n\vert^2}-1\right)dx\le C_1(\alpha_0, \rho).
\end{eqnarray*}
By Lemma \ref{GN}, there exist $C_2, C_3(q)>0$, independent of $n$, such that
\begin{eqnarray*}
 \int_{\mathbb{R}^2}\vert u_n\vert^4dx\le C_2 \rho  \int_{\mathbb{R}^2}\vert\nabla u_n\vert^2dx, \, \int_{\mathbb{R}^2}\vert u_n\vert^{2p+2}dx\le C_3(p) \rho \left(\int_{\mathbb{R}^2}\vert\nabla u_n\vert^2dx\right)^{p}.
\end{eqnarray*}
In addition, since $\{u_n\}\subset \mathcal{M}$, by Lemma \ref{mono} (iii) and $(f_4)$, it follows that
\begin{eqnarray*}
\int_{\mathbb{R}^{2}}\vert\nabla u_{n}\vert^{2}dx=\int_{\mathbb{R}^2}\left[f(u_n)u_n-2F(u_n)\right]dx\ge 2\int_{\mathbb{R}^2}F(u_n)dx\ge \frac{2\beta}{p}\int_{\mathbb{R}^2}\vert u_n\vert^pdx.
\end{eqnarray*}
Then, in view of (\ref{7-14-1}), (\ref{10-20-1}) and $p>4$, by Lemma \ref{upper-bound}, for any $\epsilon\in(0,1)$, there exists $C_{\epsilon}>0$ such that
\begin{eqnarray}\label{8-26-1}
    		&&\int_{\mathbb{R}^{2}}\vert F\left((s_0-\tilde{s}_k)\star u_n\right)\vert dx\nonumber\\
    &\le&\epsilon e^{2(s_0-\tilde{s}_k)} C_2\rho\|\nabla u_n\|_2^2+\epsilon e^{p(s_0-\tilde{s}_k)} \left(C_1(\alpha_0, \rho)\right)^{\frac{1}{2}}\left(C_3(p)\right)^{\frac{1}{2}}\rho^{\frac{1}{2}}\|\nabla u_n\|_2^p\nonumber\\
    &&+e^{(p-2)(s_0-\tilde{s}_k)} \frac{pC_{\epsilon}}{2\beta}\|\nabla u_n\|_2^2\nonumber\\
    &\le& \left[\epsilon C_2\rho+\frac{\epsilon\sigma M^{\frac{p-2}{2}}}{\left(M+1+k\right)^{p-2}}+\frac{(\frac{\pi}{2\alpha_0})^{\frac{p-4}{2}}}{\left(M+1+k\right)^{p-4}}\frac{pC_{\epsilon}}{2\beta}\right]e^{2(s_0-\tilde{s}_k)}\|\nabla u_n\|_2^2,~~
    \end{eqnarray}
 where $\sigma:=(\frac{\pi}{2\alpha_0})^{\frac{p-2}{2}} \left(C_1(\alpha_0, \rho)\right)^{\frac{1}{2}}\left(C_3(p)\right)^{\frac{1}{2}}\rho^{\frac{1}{2}}$.

Taking $k:=p+1+16\sigma^{\frac{2}{p-2}}+(32p)^{\frac{1}{p-4}}(\frac{\pi}{2\alpha_0})^{\frac{1}{2}}$ in (\ref{8-26-1}), we get, for any $\epsilon\in(0,1)$,
\begin{eqnarray*}
\frac{\epsilon\sigma M^{\frac{p-2}{2}}}{\left(M+1+k\right)^{p-2}}&\le& \frac{\sigma M^{\frac{p-2}{2}}}{\left(M+1+p+1+16\sigma^{\frac{2}{p-2}}\right)^{p-2}}\\
&\le& \frac{\sigma M^{\frac{p-2}{2}}}{\left(M+1\right)^\frac{p-2}{2}\left(p+1+16\sigma^{\frac{2}{p-2}}\right)^\frac{p-2}{2}}\\
&\le& \frac{\sigma }{\left(16\sigma^{\frac{2}{p-2}}\right)^\frac{p-2}{2}}\\
&\le&\frac{1}{16}.
\end{eqnarray*}
Taking $\epsilon:=\frac{1}{16 C_2\rho+1}$ in (\ref{8-26-1}), we obtain $\epsilon C_2\rho\le \frac{1}{16}$. Furthermore,
\begin{eqnarray*}
\frac{(\frac{\pi}{2\alpha_0})^{\frac{p-4}{2}}}{\left(M+1+k\right)^{p-4}}\frac{pC_{\epsilon}}{2\beta}&=&\frac{(\frac{\pi}{2\alpha_0})^{\frac{p-4}{2}}}{\left(M+1+k\right)^{p-4}}\frac{p(4\beta+M_{\epsilon})}{2\beta}\\
&\le&\frac{4\beta+M_{\epsilon}}{64\beta}\\
&\le&\frac{1}{8},
\end{eqnarray*}
where $\beta\ge \beta_0$ is taken larger if necessary. Thus, we deduce that
\begin{eqnarray*}
  \int_{\mathbb{R}^{2}}\vert F\left((s_0-\tilde{s}_k\right)\star u_n)\vert dx\le\frac{1}{4}e^{2(s_0-\tilde{s}_k)}\|\nabla u_n\|_2^2.
    \end{eqnarray*}	
Hence, by Lemma \ref{pro} (i), it follows that
\begin{eqnarray*}
c_{\rho}+o_n(1)&=&I(u_n)\\
&\ge&I((s_0-\tilde{s}_k)\star u_n)\\
&=&\frac{1}{2}e^{2(s_0-\tilde{s}_k)}\int_{\mathbb{R}^{2}}\vert\nabla u_n\vert^2dx-\int_{\mathbb{R}^2}F\left((s_0-\tilde{s}_k)\star u_n\right)dx\\
&\ge&\frac{1}{4}e^{2(s_0-\tilde{s}_k)}\|\nabla u_n\|_2^2,
\end{eqnarray*}
which together with Lemma \ref{energy-2} implies that, for sufficiently large $n_1$, we get
\begin{eqnarray*}
\|\nabla u_n\|_2^2\le 8c_{\rho}e^{-2(s_0-\tilde{s}_k)}
\le \beta^{\frac{-2}{p-4}}A_3 \frac{16\alpha_0}{\pi}\left(M+1+k\right)^2, \quad \forall n\ge n_1.
\end{eqnarray*}
Thus, there exists $\beta_1:=\beta(\rho, p)>0$ large enough such that $\|\nabla u_n\|_2^2\le \frac{\pi}{2\alpha_0}$ for all $n\ge n_1$ and $\beta\ge \beta_1$.

\end{proof}

{\bf Completion of proof of Theorem \ref{Thm1.1}}.
Take a minimizing sequence $\{u_n\}$ of the functional $J$ on $\mathcal{D}_{\rho}\cap\mathcal{M}$. By Lemma \ref{radial}, we may assume that $\{u_{n}\}\subset H_r^1(\mathbb{R}^2)$.
By Lemma \ref{bound}, for $n\ge n_1$ and $\beta\ge \beta_1$, we have
\begin{eqnarray}\label{5-20-1}
   \int_{\mathbb{R}^{2}}\vert \nabla u_{n}\vert^{2}dx\leq\frac{\pi}{2\alpha_0}.
   \end{eqnarray}
Then $\{u_n\}$ is bounded in $H_r^{1}(\mathbb{R}^{2})$ and there exists $u_0\in H_r^{1}(\mathbb{R}^{2})$ such that
$u_{n}\rightharpoonup u_0$ in $H_r^{1}(\mathbb{R}^{2})$, $u_{n}\rightarrow u_{0}$ in $L^{r}(\mathbb{R}^{2})$ for all $r>2$ and $u_{n}(x)\rightarrow u_0(x)$ for a.e. $x\in \mathbb{R}^{2}$.

Let $P(t)=F(t)$ and $Q(t)=e^{4\alpha_0 t^2}-1-4\alpha_0 t^2$. By $(f_1)-(f_2)$, it is easily seen that
$\frac{P(t)}{Q(t)}\to 0$ as $t\to 0$ and $\frac{P(t)}{Q(t)}\to 0$ as $t\to +\infty$.
By (\ref{5-20-1}) and Lemma \ref{TM}, we get, for some $C_1>0$,
\begin{eqnarray*}
\int_{\mathbb{R}^2}\left(e^{4\alpha_0 u_n^2}-1\right)dx=\int_{\mathbb{R}^2}\left(e^{2\pi \frac{2\alpha_0 u_n^2}{\pi}}-1\right)dx \le C_1.
\end{eqnarray*}
By the Strauss radial lemma (see \cite{BL1983-1}), there exists $C_2>0$ such that $\vert u_{n}(x)\vert\leq C_2\vert x\vert^{-1}, \forall \vert x\vert\neq 0$. Then, applying the Strauss compactness lemma (see \cite{BL1983-1}) implies that
\begin{eqnarray*}\label{5-21-1}
\int_{\mathbb{R}^2} F(u_{n})dx \to \int_{\mathbb{R}^2}F(u_0)dx.
 \end{eqnarray*}
 Similarly, taking $P(t)=f(t)t$ and $Q(t)=\vert t\vert^{p+1}(e^{4\alpha_0 t^2}-1)$, we obtain
   	\begin{eqnarray*}\label{5-21-2}
   	\int_{\mathbb{R}^{2}}f(u_{n})u_ndx\rightarrow\int_{\mathbb{R}^{2}} f(u_{0})u_0dx.
   	\end{eqnarray*}
Then, by Lemma \ref{minimizer}, it follows that
\begin{equation}\label{8-25-1}
 c_{\rho}=J(u_0)>0.
 \end{equation}

In what follows, we claim that $u_0\in \mathcal{S}_{\rho}\cap \mathcal{M}$. By contradiction, we assume that $\int_{\mathbb{R}^{2}}\vert u_0\vert^{2}dx<\rho$.
Since $\big(\mathcal{D}_{\rho}\setminus\mathcal{S}_{\rho}\big)\cap \mathcal{M}$ is an open subset of $\mathcal{M}$, there exists a Lagrange multiplier $\theta\in \mathbb{R}$ such that $u_0$ weakly solves
   	\begin{eqnarray}\label{5-22-1}
   		-\triangle u_0-f(u_0)+\theta(-\triangle u_0-\frac{1}{2}h(u_0))=0.
   	\end{eqnarray}
By classical arguments as in \cite{W1996}, $u_0$ satisfies the following Pohozaev identity
   	\begin{eqnarray}\label{5-22-2}
\int_{\mathbb{R}^{2}}\left[F(u_0)+\frac{1}{2}\theta H(u_0) \right]dx=0.
   	\end{eqnarray}
Then, using $u_0\in\mathcal{M}$, it follows that
 	\begin{eqnarray}\label{8-11-4}
\theta\int_{\mathbb{R}^{2}}\vert \nabla u_0\vert^{2}dx&=&\theta\int_{\mathbb{R}^{2}}H(u_0)dx\nonumber\\
&=&-2\int_{\mathbb{R}^{2}}F(u_0)dx\nonumber\\
&=&-\int_{\mathbb{R}^{2}}\vert \nabla u_0\vert^{2}dx+2c_{\rho},
   	 \end{eqnarray}
which combines with Lemma \ref{minimizer} implies $\theta>-1$. Together with (\ref{5-22-1}) and (\ref{5-22-2}), we obtain
   	\begin{eqnarray*}
 (1+\theta)\int_{\mathbb{R}^{2}}\vert \nabla u_0\vert^{2}dx=\int_{\mathbb{R}^{2}}\left[ H(u_0)+\frac{1}{2}\theta(h(u_0)\tilde{u}-2H(u_0)) \right]dx.
   	\end{eqnarray*}
Then, by $u_0\in\mathcal{M}$ again, we get
   	\begin{eqnarray}\label{3.4}
   	\theta\int_{\mathbb{R}^{2}}\left[h(u_0)\tilde{u}-4H(u_0) \right]dx=0.
   	\end{eqnarray}
   By the regularity arguments as in \cite{doS2001}, it follows that $u_0$ is continuous. By (\ref{3.4}) and Lemma \ref{mono} (i), we obtain that $h(u_0)u_0-4H(u_0)=0$ for all $x\in \mathbb{R}^2$. Since $u_0\in H^1(\mathbb{R}^2),$ we get $\vert u_0(x)\vert \to0$ as $\vert x\vert\to +\infty$. Thus, there exists a open interval $I$ such that $0\in \bar{I}$ and $h(u_0) u_0-4H(u_0)=0$ for $u_0\in \bar{I}$, which implies that $H( u_0)=C\vert u_0\vert^{4}$ for some $C>0$ and $u_0\in \bar{I}$. But this is contrary to $(f_1)-(f_2)$. Therefore, $\theta=0$. In view of (\ref{8-11-4}), we get $\int_{\mathbb{R}^2}F(u_0)dx=0$. However, by $(f_4),$ we have $\int_{\mathbb{R}^2}F( u_0)dx>0$, which provides a contradiction. Hence the claim holds.

Now, by standard arguments, there exist two Lagrange multipliers $\lambda,\mu\in\mathbb{R}$ such that $u_{0}$ weakly solves
   \begin{eqnarray}\label{8-24-1}
   	-\bigtriangleup u_{0}-f(u_{0})+\lambda u_{0}+\mu(-\bigtriangleup u_{0}-\frac{1}{2}h(u_{0}))=0,
   \end{eqnarray}
 by which we get
   \begin{eqnarray}\label{8-23-1}
   &&\int_{\mathbb{R}^{2}}[F(u_{0})-\frac{1}{2}\lambda u_{0}^{2}+\frac{1}{2}\mu H(u_{0})]dx=0,\\
 &&(1+\mu)\int_{\mathbb{R}^{2}}\vert \nabla u_{0}\vert^{2}dx-\int_{\mathbb{R}^{2}}[f(u_{0})u_{0}-\lambda u_{0}^{2}+\frac{1}{2}\mu u_{0}h(u_{0})]dx=0.\label{8-23-2}
   \end{eqnarray}
If $\mu=-1$, by (\ref{8-23-1}) and $u_0\in \mathcal{M},$ it follows that
   \begin{eqnarray*}
   \lambda\int_{\mathbb{R}^{2}}u_{0}^{2}dx
   &=&\int_{\mathbb{R}^{2}}\left[2F(u_{0})-H(u_{0})\right]dx\\
   &=&\int_{\mathbb{R}^{2}}2F(u_{0})dx-\int_{\mathbb{R}^{2}}\vert \nabla u_{0}\vert^{2}dx\\
   &=&-2c_{\rho}<0,
   \end{eqnarray*}
which implies $\lambda<0$. By the Strauss radial lemma in \cite{BL1983-1}, we may assume that $u_0$ is continuous. Then by (\ref{8-24-1}) it follows that $\lambda u_0(x)=f(u_0(x))-\frac{1}{2} h(u_0(x))$ holds for all $x\in \mathbb{R}^2$.
Using $(f_1)-(f_2)$, we have
   \begin{eqnarray*}
   	\lim_{s\rightarrow0}\frac{h(s)}{s}=\lim_{s\rightarrow0}\frac{H(s)}{\frac{1}{2}s^{2}}=\lim_{s\rightarrow0}\frac{sf(s)-2F(s)}{\frac{1}{2}s^{2}}=0,
   \end{eqnarray*}
which implies that $u=0$ is an isolated solution of $\lambda u=f(u)-\frac{1}{2} h(u)$. Hence $u_0\equiv0$, which is in contradiction with $u_0\in \mathcal{S}_{\rho}$. Therefore, $\mu\neq-1$. Combining with (\ref{8-23-1}) and (\ref{8-23-2}) gives
   \begin{eqnarray*}
   &&	-(1+\mu)\int_{\mathbb{R}^{2}}\vert \nabla u_{0}\vert^{2}dx+\int_{\mathbb{R}^{2}}[f(u_{0})u_{0}+\frac{1}{2}\mu u_{0}h(u_{0})]dx\\
   &=&\int_{\mathbb{R}^{2}}[2F(u_{0})+\mu H(u_{0})]dx,
   \end{eqnarray*}
which together with $u_0\in \mathcal{M}$ implies that
   \begin{eqnarray*}
   	\mu \int_{\mathbb{R}^{2}}[h(u_{0})u_{0}-4H(u_{0})]dx=0.
   \end{eqnarray*}
   By similar arguments as above we get
 $\mu=0$. Hence $u_{0}\in\mathcal{S}_{\rho}\cap\mathcal{M}$ is a weak solution of (\ref{1.1}) and $\mathcal{M}$ is a natural constraint. Thus, using (\ref{8-25-1}), $u_0$ is a normalized ground state solution of (\ref{1.1}).
Furthermore, since $u_0$ satisfies the Pohozaev identity it follows that $\lambda>0$.
The proof is complete.

\bmhead{Acknowledgments}
The research of Xiaojun Chang is supported by National Natural Science Foundation of China (No.11971095), while Duokui Yan is supported by National Natural Science Foundation of China (No.11871086). This work was done when Xiaojun Chang visited the Laboratoire de Math\'ematiques, Universit\'e de Bourgogne Franche-Comt\'e during the period from 2021 to 2022 under the support of China Scholarship Council (202006625034), and he would like to thank the Laboratoire for their support and kind hospitality.

\section*{Declarations}

\begin{itemize}
\item Conflict of interest: The authors declare that they have no competing interests.
\end{itemize}

\end{document}